\documentclass{proc-l}

\usepackage{amsmath, amssymb, amscd, amsthm, amsfonts, amstext,
  amsbsy, enumerate, mathrsfs}

\newcommand{\seqo}{{}^{<\omega}\omega}
\newcommand{\fhi}{\varphi}
\newcommand{\imp}{\Rightarrow}
\renewcommand{\L}{\mathcal{L}}
\newcommand{\C}{\mathcal{C}}
\newcommand{\E}{\mathcal{E}}
\newcommand{\RR}{\mathbb{R}}
\newcommand{\EE}{\mathbb{E}}
\newcommand{\Linf}{\L_{\omega_1 \omega}}
\newcommand{\id}{\operatorname{id}}
\newcommand{\dom}{\operatorname{dom}}
\newcommand{\Mod}{\operatorname{Mod}}
\newcommand{\rp}{\operatorname{rp}}

\newtheorem{theorem}{Theorem}[section]
\newtheorem{lemma}[theorem]{Lemma}
\newtheorem{corollary}[theorem]{Corollary}
\newtheorem{proposition}[theorem]{Proposition}
\newtheorem*{question}{Question}
\newtheorem*{problem}{Problem}

\theoremstyle{definition}
\newtheorem{defin}[theorem]{Definition}
\newtheorem{claim}{Claim}[theorem]

\theoremstyle{remark}
\newtheorem{remark}[theorem]{Remark}

\numberwithin{equation}{section}

\begin{document}

\title[On the complexity of isomorphism and bi-embeddability]{On the complexity of the relations of isomorphism and bi-embeddability}

\author{Luca Motto Ros}
\address{Albert-Ludwigs-Universit\"at Freiburg \\
Mathematisches Institut -- Abteilung f\"ur Mathematische Logik\\
Eckerstra{\ss}e, 1 \\
 D-79104 Freiburg im Breisgau\\
Germany}
\email{luca.motto.ros@math.uni-freiburg.de}
\thanks{The author would like to thank the FWF (Austrian Research  
  Fund) for generously supporting this research through  Project  
  number P 19898-N18.}


\date{\today}

\commby{Julia Knight}

\begin{abstract}
 Given an $\Linf$-elementary class $\C$, that is the
 collection of the countable models of some $\L_{\omega_1
   \omega}$-sentence, denote by $\cong_\C$ and $\equiv_\C$ the
 analytic equivalence relations of, respectively, isomorphism and
 bi-embeddability on $\C$. Generalizing some questions of Louveau and
 Rosendal \cite{louros}, in \cite{friedmanmottoros} it was proposed the
 problem of determining which pairs of analytic equivalence relations
 $(E,F)$  can be realized (up to Borel bireducibility) as pairs of the
 form  $( \cong_\C,
 \equiv_\C)$, $\C$ some $\Linf$-elementary class (together with
 a partial answer for some specific cases). Here we will 
 provide an almost complete solution to such problem: under very
 mild conditions on
 $E$ and $F$, it is always possible to find such an $\Linf$-elementary
 class $\C$.
\end{abstract}

\maketitle

\section{Introduction}

An equivalence relation $E$ defined on a Polish space (or, more
generally, on  
a standard Borel space) $X$ is said to be  \emph{analytic} if it is an
analytic  
subset of $X \times X$. Analytic equivalence relations arise very often in
various areas of mathematics, and are usually connected with important
classification problems --- see e.g.\ the preface of
\cite{hjorth} for a brief but 
informative introduction to this subject. The most popular way to
measure the relative 
complexity of two analytic equivalence relations $E$ and $F$ is given by the
notions of \emph{Borel reducibility} and \emph{Borel bireducibility} (in
symbols $\leq_B$ and $\sim_B$, respectively): $E \leq_B F$ if  there is a
Borel function $f$ between the corresponding domains which reduces $E$
to $F$, that is 
such that $x\, E\, y \iff f(x)\, F\, f(y)$ for every $x,y$ in the domain of
$E$, and $E \sim_B F$ if  $E \leq_B F$ and $F \leq_B E$. (We will 
denote by $<_B$ the  strict part
of $\leq_B$.) Intuitively, $E \leq_B F$  means that $E$ is not more
complicated than 
$F$, so that  $E \sim_B F$ means that $E$ and $F$ have same complexity.

Similar definitions and terminology will be also applied to
\emph{analytic quasi-orders} (i.e.\ reflexive and transitive relations
$R$ on a
standard Borel space $X$ which are analytic subsets of $X \times X$), and when dealing with an analytic quasi-order $R$ we will also often consider the analytic equivalence
relation $E_R = R \cap R^{-1}$ canonically induced by $R$.

A nice example of analytic equivalence relation is  the following:
consider an $\Linf$-elementary class $\C$, that is the collection of
all countable models of some sentence of the
infinitary logic $\L_{\omega_1 \omega}$, $\L$ some countable
language. Assuming that all these models have domain $\omega$ (the set
of natural numbers), we can canonically identify each of them with an
element of the Polish space of $\L$-structures $\Mod(\L)$ (which is homeomorphic to the Cantor space), and by a well-known theorem
of Lopez-Escobar (see e.g.\ \cite[Theorem 16.8]{kechris}), $\C \subseteq
\Mod(\L)$ is an $\Linf$-elementary class if and only if $\C$ is
Borel and invariant under isomorphism: this easily implies that the
relation of 
isomorphism $\cong_\C$ between elements of $\C$  becomes an analytic
equivalence 
relation (relations of this form 
will be simply called \emph{isomorphism relations}). 

If in the previous definition we replace
isomorphisms with (logical) embeddings between elements of $\C$, we
get the analytic quasi-order $\sqsubseteq_\C$ of embeddability on
$\C$, which in turn 
canonically induces the analytic equivalence relation $\equiv_\C$ of
bi-embeddability between elements of $\C$. The possible relationships
between 
$\cong_\C$ and $\equiv_\C$ were first investigated in
\cite{friedmanmottoros}, where the authors constructed
various $\Linf$-elementary classes $\C$ satisfying certain conditions on
$\cong_\C$ and $\equiv_\C$ to answer  some questions posed by
Louveau and Rosendal in their \cite{louros}: in particular, in \cite{friedmanmottoros} it is shown that given an arbitrary analytic equivalence relation $F$ there is an $\Linf$-elementary class $\C$ such that ${\cong_\C} \sim_B \id(\RR)$, where $\id(\RR)$ denotes the identity relation on $\RR$, and ${\equiv_\C} \sim_B F$. After those examples, the 
following problem was formulated: 

\begin{problem}
Consider an arbitrary pair of analytic
equivalence relations $(E,F)$: is it possible to find an $\L_{\omega_1
  \omega}$-elementary class $\C$ such that $E \sim_B {\cong_\C}$ and
$F \sim_B {\equiv_\C}$?

 Similarly, one can consider the analogous
question regarding a pair $(E,R)$ consisting of an analytic
equivalence relation and an analytic quasi-order: 
is there an
$\Linf$-elementary class $\C$ such that $E \sim_B {\cong_\C}$ and $R
\sim_B {\sqsubseteq_\C}$?
\end{problem}

For ease of exposition, if such a $\C$ exists we will say that $\C$ \emph{represents} the pairs $(E,F)$ or $(E,R)$, respectively.
The problem of giving a complete and general characterization of those $(E,F)$ and $(E,R)$ which can be represented by an $\Linf$-elementary class was considered in \cite{friedmanmottoros} a potentially difficult problem. First we must notice that there are some obvious limitations to the possibility of having such a representation: for example,
since there are
  many analytic equivalence relations which are not even Borel
  reducible to an isomorphism relation, we should at least 
ask that $E$ is  a \emph{quasi-isomorphism relation}, i.e.\
that $E$ is Borel bireducible with some isomorphism
relation on some $\Linf$-elementary class (by \cite[Theorem 5.5.1]{hodges}, such a class can be assumed to always consist of connected graphs) or, equivalently, to an equivalence
relation induced by the Borel 
action of a closed subgroup of the symmetric group $S_\infty$ --- see \cite[Theorems 2.3.5 and 2.7.3]{beckerkechris}. In
contrast, no \emph{a priori} condition must be put on $F$ or $R$ since in
\cite{friedmanmottoros} it is shown that any analytic equivalence
relation (resp.\ any analytic quasi-order) is actually Borel bireducible with the
bi-embeddability (resp.\ embeddability) relation on a corresponding
$\Linf$-elementary class --- see Theorem \ref{theorfrimot}. 

A less trivial, but still easy, restriction that must be put on the
pairs $(E,F)$ and $(E,R)$ is given by the
following ``cardinality'' consideration. Denote by $\id(n)$,
  where $1 \leq n \leq \omega$, an arbitrary analytic equivalence
  relation with exactly $n$ classes. Given
  an $\L_{\omega_1 
  \omega}$-elementary class $\C$, since $\equiv_\C$ is by definition
coarser than 
$\cong_\C$  then the ``cardinality'' of $\C /_{\equiv_\C}$ cannot
exceed the ``cardinality'' of $\C /_{\cong_\C}$, that is:
\begin{itemize}
\item if ${\cong_\C} \leq_B \id(\omega)$ then ${\equiv_\C} \leq_B
  {\cong_\C}$;
\item if $F \leq_B {\equiv_\C}$ then $F \leq_B {\cong_\C}$, where $F$
  is one of $\id(1), \dotsc, \id(\omega), \id(\RR)$.
\end{itemize}
 As we will see, if 
Vaught's Conjecture is true (equivalently, by the Silver's
dichotomy, if every $E$ which is
a quasi-isomorphism relation is $\leq_B$-comparable with
$\id(\RR)$) then 
these are quite 
surprisingly  the unique obstructions to
get a representation of the pairs $(E,F)$ and $(E,R)$.
 In fact, Theorems \ref{theorRcompatible} and
\ref{theor'} (which constitute the
main results of the paper) show that given a quasi-isomorphism relation $E$ and an analytic quasi-order $R$ such that either $\id(\RR) \leq_B E$ or $E_R \leq_B E, \id(\RR)$, there exists an $\Linf$-elementary class $\C$ with the property that ${\cong_\C} \sim_B E$ and ${\sqsubseteq_\C} \sim_B R$. In particular, if $E$ is a quasi-isomorphism relation and either $E
 \leq_B {\id(\omega)}$ or ${\id(\RR)} \leq_B E$, then there is an $\Linf$-elementary class $\C$ representing the pair 
$(E,F)$ (resp.\ $(E,R)$) \emph{if and only if} 
either  $F \leq_B E$ (resp.\ $E_R \leq_B E$) or ${\id(\RR)} \leq_B
E$. 
The same kind of
result can be obtained by considering homomorphisms,
weak-homomorphisms or weak-epimorphisms instead of embeddings --- see
Sections \ref{sectionEsimple} and \ref{sectionepi}.

 The results above can also be
na\"ively interpreted as a proof  that the complexities of the
relations of isomorphism and 
bi-embeddability on some $\Linf$-elementary class are 
(almost) independent from each other: given any isomorphism relation $\cong_\E$  and an arbitrary
quasi-order $R$ on 
$\E$ such that ${\cong_\E} \subseteq E_R$ (so that $R$ can
\emph{potentially} 
be the embeddability relation on $\E$), then the above mentioned
results show that unless both $\cong_\E$ and $E_R$ are
$\leq_B$-incomparable 
with $\id(\RR)$ there
is an $\Linf$-elementary class $\C$ such that\footnote{Our result is even stronger: in fact we get that there is a pair of functions \emph{simultaneously} witnessing both ${\cong_\C} \simeq_{cB} {\cong_\E}$ (see Section \ref{sectioncomparison} for the definition of $\simeq_{cB}$) and ${\sqsubseteq_\C} \sim_B R$.} ${\cong_\C} \sim_B {\cong_\E}$ and ${\sqsubseteq_\C} \sim_B R$. This also means that almost all the possible mutual
relationships between 
the isomorphism and the (bi-)embeddability relations can actually be
realized with a suitable $\Linf$-elementary class.

On the way of proving our main result, we will deal with the notion of Borel
isomorphism, which  plays  a key role in the
proofs of the results of
this paper. This notion (which is strictly finer
  than Borel bireducibility) slightly strengthens some variants of Borel reducibility already introduced in \cite{friedmanstanley}, and we feel that the applications we are going to present can be viewed as evidence that such a notion is natural, interesting and useful in the study of analytic equivalence relations and quasi-orders.

The paper is organized as follows.  In Section \ref{sectioncomparison} we will prove some basic results about
classwise Borel isomorphism and classwise Borel embeddability which will be useful in the subsequent sections (but which may also be of independent interest). In 
Section \ref{sectionEsimple} we will prove the main results of this
paper (Theorems
\ref{theorRcompatible} and  \ref{theor'}), and finally in Section
\ref{sectionepi} we will 
show how to extend the results of the previous section to the case
of weak-epimorphisms.

We  assume that the reader is quite
familiar with the standard terminology and basic results about analytic
equivalence relations and Borel reducibility: references for these
topics are for example 
\cite{kechris}, \cite{beckerkechris}, \cite{hjorth} and \cite{gao}. Part of the main
techniques that will be used in the paper were first introduced in
\cite{friedmanmottoros} and, 
partially, in \cite{friedmanstanley}:  for the
reader's convenience, throughout the paper we will
recall the main results and constructions  coming from those papers, but we refer to the original works for proofs and
detailed explanations.

\section{Classwise Borel isomorphism and classwise Borel embeddability}\label{sectioncomparison}

The present section contains
a basic analysis of  Borel
isomorphism and classwise Borel embeddability (see Definitions
\ref{defBorelisomorphism} and
\ref{defBorelembedding}), and is mainly motivated by the fact that 
some of the properties presented here will be used in
Theorem \ref{theorRcompatible}. Nevertheless, the results of this section are
also interesting \emph{per se} as they constitute a study of some
basic properties of these natural and useful notions, and a modest
contribution to the study of orbit equivalence relations (i.e.\ of
those analytic equivalence relations which are induced by a Borel
action of a Polish group on some standard Borel space).

\begin{defin}[\cite{friedmanmottoros}]
  \label{defBorelisomorphism} 
Let $E,F$ be two analytic equivalence relations on standard Borel
spaces $X,Y$, respectively. We say that $E$ is
\emph{classwise Borel isomorphic}\footnote{This notion was introduced in \cite{friedmanmottoros} as ``Borel isomorphism'': however, in the present paper we decided to adopt a variation of that name to avoid confusion with a different notion of Borel isomorphism which is now quite standard in the literature, see e.g.\ \cite{gao}.} 
  to $F$ ($E \simeq_{cB} F$ in symbols)  if there are Borel reductions
  $\fhi \colon X \to Y$ and $\psi \colon Y \to X$ of 
  $E$ into $F$ and $F$ into $E$, respectively,  such that their
  factorings to the 
  quotient spaces $\hat{\fhi} \colon
  X/_E \to Y/_F$ and $\hat{\psi} \colon Y/_F \to X/_E$ are
  bijections and satisfy $\hat{\fhi} = \hat{\psi}^{-1}$.

In other words: $E \simeq_{cB} F$ if and only if there is a bijection $f \colon X/_E \to Y/_F$ such that both $f$ and $f^{-1}$ admit Borel liftings.
\end{defin}

 Classwise Borel isomorphism \emph{strictly refines}
  Borel bireducibility: a first example of this phenomenon was given in
\cite{friedmanmottoros} by considering the $\sim_B$-equivalence
class of $\id(\RR)$, but this result will be extended in Theorem
\ref{theornonborelisomorphic} to the
$\sim_B$-equivalence class of \emph{any} orbit equivalence relation which
Borel reduces $\id(\RR)$. However, as we will see in Theorems
\ref{theorBorelorbitequivrel} and \ref{theorRcompatibility}, the two notions coincide if we restrict
our attention to
some special case, like e.g.\ to the class of Borel orbit equivalence
relation.

\begin{defin}\label{defBorelembedding}
Given two analytic equivalence
relations $E,F$ on standard Borel spaces $X,Y$, respectively, we say that $E$ \emph{classwise Borel embeds} into $F$
($E \sqsubseteq_{cB} F$ 
in symbols) just in case there is a Borel $F$-saturated subset $Y_E
\subseteq Y$ such that $E \simeq_{cB} {F \restriction
Y_E}$.
\end{defin}

The notion of classwise Borel embeddability is not far from what
  in \cite{gao} is called \emph{faithful Borel reducibility} (a notion first introduced in \cite{friedmanstanley}). However,
  classwise Borel embeddability is a strictly stronger notion because we require the
  existence of a sort of ``inverse'' 
(modulo equivalence classes) of the
 reduction from $E$ to $F$.

Another interesting property of classwise Borel embeddability is that many
popular classes of analytic equivalence relations are closed under this
notion of reducibility: if e.g.\ we consider the class of isomorphism
relations, or even the broader class of orbit equivalence relations, we
have that 
any analytic equivalence relation which classwise Borel embeds into an element
of this class is actually classwise Borel isomorphic to an element of the same
class. This should be contrasted with the fact that it is still an
open (and seemingly hard) problem to determine if every analytic
equivalence relation which 
is Borel
reducible to an element of one of the above mentioned  classes is
Borel bireducible with a member of the same class: to the best of our knowledge,  \cite[Theorem 11.3.9]{gao} and Remark \ref{remclosure} are the unique results in this
direction. 

 Finally, a classwise Borel embedding between two isomorphism relations $\cong_\C$ and $\cong_\E$ can be seen as a sort of \emph{$\Linf$-interpretation} between the two
 elementary classes in the sense explained in \cite[p.\ 897]{friedmanstanley}, that is in the sense that it provides a correspondence between $\Linf$-sentences. More precisely, for every $\Linf$-sentence $\Phi$ there is an $\Linf$-sentence $\Psi$ such that the set of models in $\C$ satisfying $\Phi$ is mapped by the witnesses of ${\cong_\C} \sqsubseteq_{cB} {\cong_\E}$ into the set of models in $\E$ which satisfy $\Psi$ in a bijective (up to isomorphism) and Borel way.

The next proposition shows that classwise Borel embeddability is the counterpart in
terms of reducibility of classwise Borel isomorphism.

\begin{proposition}\label{propSB}
Let $E,F$ be analytic equivalence relations. If $E \sqsubseteq_{cB} F$ and
$F \sqsubseteq_{cB} E$ then $E \simeq_{cB} F$. Therefore
$E \simeq_{cB} F$ if and only if $E \sqsubseteq_{cB} F
\sqsubseteq_{cB} E$.
\end{proposition}

\begin{proof}
It is enough to notice that we can apply the usual Schr\"oder-Bernstein argument because if $\fhi \colon \dom(E) \to \dom(F)$ and $\psi \colon \dom(F) \to \dom(E)$ witness $E \sqsubseteq_{cB} F$, then for every $E$-invariant Borel $A \subseteq \dom(E)$ the $F$-saturation of $\fhi(A)$ is $\psi^{-1}(A)$, hence a Borel set (and similarly exchanging $E$ and $F$). 
\end{proof}

It is not hard to see that if $E$ is a countable analytic equivalence
relation on $X$ and $F$ is an arbitrary Borel equivalence relation on
$Y$,  then $E
\leq_B F$ implies $E \sqsubseteq_{cB} F$ (the converse to this fact is
obvious). In fact, if $f$ is a Borel reduction of $E$ into $F$, consider the
 Borel set $Z = \{ (x,y) \in X \times Y \mid f(x)\, F\, y \}$ and the map
 $g \colon Z \to Y \colon (x,y) \mapsto y$: by countability of $E$, the
 Borel map $g$ is 
  countable-to-$1$, 
 so ${\rm range}(g)$ 
 (which is the $F$-saturation of ${\rm range}(f)$) is Borel as well,
 and there is a Borel right inverse $h$ of $g$. Therefore, ${\rm
   range}(g)$, $f$ and the composition of $h$ with projection on the
 fist coordinate witness $E \sqsubseteq_{cB} F$. This easy observation can be
 extended with a completely different and more difficult argument to the case
 of an arbitrary orbit equivalence relation $E$, see \cite[(Proof of)
 Corollary 5.2.4]{gao}.

\begin{proposition}[Gao]\label{propgao}
Let $E$ be an  orbit equivalence relation and $F$ be an arbitrary
Borel equivalence relation. Then ${E \leq_B F} \iff E \sqsubseteq_{cB} F$.
\end{proposition}

As a corollary of Propositions \ref{propgao} and \ref{propSB}, we get that for Borel
orbit equivalence 
relations the notions of Borel bireducibility and classwise Borel isomorphism
coincide.

\begin{theorem}\label{theorBorelorbitequivrel}
If $E,F$ are Borel orbit equivalence relations, then $E \sim_B F$ if and
only if $E \simeq_{cB} F$.
\end{theorem}

Notice that the results above cannot be extended to arbitrary orbit equivalence relations: in fact in \cite[Theorem 4]{gao2} it is proved that e.g.\ the relation of isomorphism on countable graphs does not classwise Borel embeds into (in fact it does not even faithful Borel reduce to) the relation of isomorphism on countable linear orders (or on ``simple'' countable trees), whereas all these isomorphism relations are $S_\infty$-complete (and hence pairwise Borel bireducible) by \cite[Theorem 5.5.1]{hodges} and \cite[Theorems 1 and 3]{friedmanstanley}, respectively.

One of the main limitations of Proposition \ref{propgao} is obviously that
the equivalence relation $F$ must be Borel. Our next goal will be to
show that in some specific situations (that is for some specific orbit
equivalence relations $E$) such restrictions can be removed,
albeit in this case we have  to compensate for this with the
requirement that $F$ is
an orbit equivalence relation as well.  In the terminology of
\cite{beckerkechris}, $(Y,a)$ is an \emph{effective Borel $G$-space}
if $Y$ is a $\Delta^1_1$ subset of a recursively presented Polish
space, $G$ is a recursively presented Polish group with recursive composition
and inverse functions, and $a$ is a $\Delta^1_1$ action of $G$ on
$Y$. Moreover, if $y$ is any element of $Y$ we denote by $\omega_1^y$ the first
(countable) ordinal not recursive-in-$y$, and put $\omega_1^{Gy} =
\inf \{ \omega_1^z \mid z \, F_a \, y \}$, $F_a$ being the orbit equivalence
relation  induced by $a$ on $Y$. 

\begin{lemma}\label{lemmacomeagre}
Assume $E$ is an arbitrary equivalence relation on the Polish
space $X$, and $F$ is an orbit equivalence relation on a standard
Borel space $Y$. If $E \leq_B F$ then there is an $E$-invariant Borel
comeagre set $C 
\subseteq X$ and an $F$-invariant Borel set $A \subseteq Y$ such that
$F \restriction A$ is a Borel equivalence relation and
$E \restriction C \leq_B
 F \restriction A$.
\end{lemma}

\begin{proof}
Assume that $F= F_a$ is  induced
 by the Borel action $a$ of the Polish group $G$ on  the standard Borel space $Y$,
 and let $f$ be a Borel reduction of $E$ into $F$.
Assume that $(Y,a)$ is an effective Borel $G$-space and $X$ a recursively presented Polish space (otherwise we
relativize),  and let
$p$ be a parameter such that $f$ is a $\Delta^1_1(p)$-function. 

We claim that $\{ x \in X \mid \omega_1^{(f(x),p)} \leq \omega_1^p
\}$ is comeagre in $X$. Granting this, $C = \{ x \in X \mid \omega_1^{Gf(x)}
\leq \omega_1^p \}$ is an $E$-invariant comeagre subset of $X$. Put $A = \{ y \in Y
\mid \omega_1^{Gy} \leq 
\omega_1^p \}$. Then $A$ is Borel and $F$-invariant (so that $C =
f^{-1}(A)$ is Borel as well), and by the
relativized version of \cite[Proposition 7.2.2]{beckerkechris} $F
\restriction A$ is Borel. But by definition of $C$, $f
\restriction C$ is a Borel map witnessing $E \restriction C \leq_B F
\restriction A$.

It remains to prove the claim. Assume toward a contradiction that
$B=\{ x \in X \mid \omega_1^{(f(x),p)} > \omega_1^p 
\}$ is nonmeagre in $X$. Since $B$ is a $\Pi^1_1(p)$ subset of $X$, by
the Sacks-Tanaki Basis Theorem for nonmeagre $\Pi^1_1(p)$ sets (see
\cite[Exercise 4F.20]{mosch}) there would be a $\Delta^1_1(p)$-point
$x_0 \in B$. But then $f(x_0)$ would be a $\Delta^1_1(p)$-point as
well, so that $\omega_1^{f(x_0),p} \leq \omega_1^p$, a contradiction! 
\end{proof}

\begin{remark} \label{remclosure}
In particular, if an arbitrary equivalence relation $E$ on a Polish space $X$ is Borel reducible to an orbit equivalence relation (resp.\ an isomorphism relation, or a countable equivalence relation) then there is an $E$-invariant Borel comeagre set $C \subseteq X$ such that $E \restriction C$ is Borel and is Borel bireducible with (in fact, classwise Borel isomorphic to) an orbit equivalence relation (resp.\ to an isomorphism relation, or to a countable equivalence relation).
\end{remark}

Call an analytic equivalence relation $E$ on the standard Borel space
$X$ \emph{invariant by comeagre subsets} if for every comeagre $C
\subseteq X$ one has $E \leq_B E \restriction C$ (note that
it is enough to restrict the attention to Borel
comeagre sets $C$). Examples of
invariant by comeagre subsets analytic equivalence relation are the
following:
\begin{itemize}
\item $E = \id(\RR)$: this is because it is a classical result that any comeagre subset of $\RR$ must contain a perfect subset;
\item $E = E_0$: by a classical fact (see e.g.\ \cite[Theorem
 3.2 ]{hjorth}), for every comeagre $C \subseteq {\rm dom}(E_0)$ we have $E_0 \restriction C \nleq_B \id(\RR)$
  (otherwise $E_0 \leq_B \id(\RR)$). But $E_0 \restriction C$
  is obviously Borel, so $E_0 \leq_B E_0 \restriction C$ by
  the Glimm-Effros Dichotomy (see e.g.\ \cite[Theorem 3.4.2]{beckerkechris}).
\end{itemize}

On the contrary, Lemma \ref{lemmacomeagre} implies, in particular,  that if $E$ is an orbit equivalence relation which is \emph{not} Borel then it cannot be invariant by comeagre subsets.

\begin{theorem}\label{theorRcompatibility}
 Let $E,F$ be orbit equivalence relations and $E$ be invariant by
 comeagre subsets. Then ${E \leq_B}
 F \iff {E \sqsubseteq_{cB} F}$. In particular, the result holds with $E = \id(\RR)$ and $E = E_0$.
\end{theorem}

\begin{proof}
One direction is obvious. For the other direction,
apply Lemma \ref{lemmacomeagre} to $E$ and $F$, use the fact that $E$ is invariant by comeagre subsets,
 and then apply Lemma
\ref{propgao}  to $E$ and $F \restriction A$ to get $E
\sqsubseteq_{cB} F \restriction A$: since $A$ is $F$-invariant
and Borel, this means $E \sqsubseteq_{cB} F$ as well.
\end{proof}

Theorem \ref{theorRcompatibility} shows, in particular, that if $E$ is
either $\id(\RR)$ or $E_0$ and $F$ is an orbit equivalence relation
then $E \leq_B F \iff E \sqsubseteq_{cB} F$. We are now going to
show that in this case there are other natural conditions which
are equivalent to the 
previous ones. 
Such conditions arise from the natural idea of considering
disjoint unions of analytic equivalence relations: given $E,F$ on
standard Borel spaces $X,Y$, respectively, we denote by $E \sqcup F$ the
analytic equivalence
relation on $X \sqcup Y$ (where $\sqcup$ denotes disjoint union) defined  by $x\, (E \sqcup F)\, y$ if and only if either
$x,y \in X$ and $x\, E\, y$, or else $x,y \in Y$ and $x\, F\, y$. Disjoint union seems a natural operation to be considered because if $E,F$ belong
to some natural class of equivalence relations (such as isomorphism
relations, orbit equivalence relations, and so on), then $E \sqcup F$ is an
analytic equivalence relation in the same class in which both $E$ and
$F$ classwise Borel embed.

\begin{proposition}\label{propRsum}
Let $E$ be an orbit equivalence relations and $F$ be either $\id(\RR)$ or $E_0$. Then the following are equivalent\footnote{Since in both cases $F \sqcup F \simeq_{cB} F$ by Proposition
  \ref{propgao} and the fact that $F \sqcup F \sqsubseteq_{cB} F$, in this proposition we could replace  all occurrences of $\leq_B$ and
  $\sim_B$ with, respectively, $\sqsubseteq_{cB}$ and $\simeq_{cB}$.}: 
\begin{enumerate}[i)]
\item $F \leq_B E$;
\item $F \sqsubseteq_{cB} E$;
 \item $E \sim_B {E' \sqcup F}$ for some analytic equivalence relation $E'$; 
\item $E \sim_B E \sqcup F$.
\end{enumerate}
\end{proposition}

\begin{proof}
i) $\imp$ ii) by Theorem
\ref{theorRcompatibility}. ii)
 $\imp$ iii) because if $Y_F \subseteq {\rm dom}(E) = Y$ is Borel,
 $E$-invariant, and such that $F \simeq_{cB} E \restriction Y_F$, then
 iii) is obviously satisfied with $E' = E \restriction (Y \setminus Y_F)$. 
Let now $E'$ witness iii): since clearly $F \sqcup F \leq_B F$, then ${E \sqcup F} \sim_B {E' \sqcup F \sqcup F} \sim_B {E' \sqcup F}
\sim_B E$, so iv) holds.
Finally,  iv) $\imp$ i) because $F \leq_B E \sqcup F \sim_B E$.
\end{proof}

Proposition \ref{propRsum} allows us to extend the example given in
\cite{friedmanmottoros} of a pair of analytic equivalence relations
which are Borel bireducible but not classwise Borel isomorphic to the context of
arbitrary orbit equivalence relations (this result should also be contrasted
with Theorem \ref{theorBorelorbitequivrel} above).

\begin{theorem}\label{theornonborelisomorphic}
Let $E$ be an orbit equivalence relation such that $\id(\RR)
\leq_B E$. Then there is an analytic equivalence relation $F \sim_B E$
such that  $F \not\simeq_{cB} E$.
\end{theorem}

\begin{proof}
Let $X$ be the domain of $E$, and let $B \subseteq \RR \times \RR$ be
a Borel set with nonempty vertical sections and with no 
Borel uniformization. Consider the Borel
equivalence relation $E_B$ on $B$ given by the vertical sections of $B$: we claim that $F = E \sqcup E_B$ works. 
Clearly $E
\leq_B E \sqcup E_B$. Moreover, since the projection on the first
coordinate witnesses $E_B \leq_B {\id(\RR)}$, we have $E \sqcup E_B \leq_B {E
  \sqcup {\id(\RR)}}$. But by Proposition 
\ref{propRsum} $E \sqcup {\id(\RR)} \leq_B E$, whence $E \sqcup E_B
\leq_B E$.

Finally, assume towards a contradiction that $E$ and $E \sqcup E_B$ are classwise
Borel isomorphic, and let $\fhi \colon X \to X \sqcup B$ and $\psi
\colon X \sqcup B \to X$ be witnesses
of this fact. The set $X' = \fhi^{-1}(B)$ is Borel and $E$-saturated, so
that $E' = E \restriction X'$ is a Borel orbit equivalence relation. Moreover the composition
of $\fhi$ with the projection on the first coordinate shows that $E'$
is smooth, so by a theorem of Burgess (see e.g.\ \cite[Corollary
5.4.12]{gao}) there is a Borel selector $s \colon X' \to X'$ for $E'$. This
implies that $f = \fhi \circ s \circ (\psi \restriction B) \colon B \to B$ is a
well-defined Borel function, and that ${\rm range}(f) = \{ b \in B
\mid b=f(b) \}$ is a Borel uniformization of $B$, a contradiction!
\end{proof}

Given two analytic equivalence relations $E,F$, say that $E$
\emph{essentially refines}\footnote{This condition is potentially
  stronger than just requiring 
  the existence of an analytic equivalence relation $E' \supseteq E$
  such that $E' \sim_B F$: however, if $F$ is a countable analytic
  equivalence relation then the two notions actually coincide.} $F$ if and only if there is an analytic
equivalence relation $E' \supseteq E$ such that $E' \simeq_{cB} F$.
The following technical result will be used in
the next section.

\begin{proposition} \label{propessentiallyrefines}
Let $E$ be an orbit equivalence relations and $F$ be either $\id(\RR)$ or $E_0$. If $F \leq_B E$ then $E$ essentially refines $F$. 
Moreover, the converse holds if $F = \id(\RR)$.
\end{proposition}

\begin{proof}
Under our assumption, $F \sqsubseteq_{cB} E$ by Proposition
\ref{propRsum}. Let $X$ be the domain of $E$ and $X_F \subseteq X$ be
Borel $E$-invariant and such that $F \simeq_{cB} E \restriction
X_F$. Define $E' \supseteq E$ on $X$ by letting $x \, E' \, y$ if and only if
either $x,y \notin X_F$ or $x \, E \, y$. Then $F \sqsubseteq_{cB} E'$ and
$E' \sqsubseteq_{cB} F \sqcup \id(1) \sqsubseteq_{cB} F$, so that $E' \simeq_{cB} F$
by Proposition \ref{propSB}.
The extra fact about $\id(\RR)$ follows from the fact that any witness of
$ \id(\RR) \leq_B E'$ witnesses $\id(\RR) \leq_B E$ as well.
\end{proof}

\section{The main result}\label{sectionEsimple}

In this section we will show our main result: if $E$ is a
quasi-isomorphism relation and $R$ is an analytic quasi-order such that at least one of
$E,E_R$ has either countably or
perfectly many equivalence classes, then there is an $\Linf$-elementary class representing the pair
$(E,R)$ if and only if either $E_R \leq_B \id(\RR),E$ or $\id(\RR)
\leq_B E$. Notice that this implies the corresponding result for pairs
$(E,F)$ consisting of two analytic equivalence
relations, therefore from this point onward we will just consider the
case of pairs of the form $(E,R)$.
 
We first consider the basic case, namely when $E_R \sim_B \id(1)$.
From this point onward, $\hat{\L}$ will denote the language of graphs consisting of just one binary relation symbol, while $\L$ will denote the language of \emph{ordered graphs}, that is a language consisting of two binary relation symbols: in particular, the interpretation of the second symbol (in a certain structure) will be always called \emph{order (relation)} of the structure\footnote{This is a little abuse of terminology: in fact, as we will see, such relation will be required to be just a transitive relation, and not an order in the usual sense.}. Finally, an \emph{ordered set-theoretical tree} is a set-theoretical tree with an extra transitive (binary) relation on its nodes.

\begin{theorem}\label{theorequivsimple}
 Let $E$ be a quasi-isomorphism
 relation. Then there is an $\L_{\omega_1 \omega}$-elementary class
 $\C$ consisting of ordered set-theoretical trees whose order is an equivalence relation (so that, in particular, it is reflexive) such that $E \sim_B {\cong_\C}$ and
 ${\sqsubseteq_\C} = {\equiv_\C} \sim_B {\id(1)}$ (in fact, if $E$ itself is an isomorphism relation then $E \simeq_{cB} {\cong_\C}$).
\end{theorem}

\begin{proof}
 Let $\hat{\C}$ be an arbitrary $\hat{\L}_{\omega_1 \omega}$-elementary class
 such that ${\cong_{\hat{\C}}} \sim_B E$. For every $x \in
 \hat{\C}$, construct the set theoretical tree $\hat{T}_x$ on $\seqo \sqcup
 \omega$  in the following way:
 consider the tree $\seqo$ with the inclusion relation. For any $s
 \in \seqo$, let $s^\EE$ be the sequence $\langle s(2i) \mid 2i < |s|
 \rangle$. If $s \neq
 \emptyset$, denote  by $\rp(s)$ the pair $(s(n),s(m))$ (also
 called \emph{relevant pair} of $s$), where
 $n,m$ are such that $|s| = \langle n,m \rangle +1$ and $\langle \cdot, \cdot \rangle$ is any bijection between $\omega \times \omega$ and $\omega$ such that $n,m \leq \langle n,m \rangle$. Now adjoin a new
 unique terminal successor taken from $\omega$ to $s$ just in case
 either $|s|$ is even, or else $|s|$ is odd and $\rp(s^\EE)$ is an edge  of $x$, 
 ensuring that at the end of this process every $n \in \omega$ is linked to some $s \in
 \seqo$. Finally, define $T_x$ by adjoining the following equivalence
 relation $E_x$ (which actually is independent from $x$) on the nodes of $\hat{T}_x$: 
\[ s\, E_x\, t  \iff (s,t \notin \seqo) \vee (s=t=\emptyset) \vee (s,t \in \seqo \setminus \{ \emptyset \} \wedge \rp(s^\EE) =
 \rp(t^\EE)). \]

Following \cite[proof of Theorem 1.1.1]{friedmanstanley}, one can easily check that
${x \cong y} \iff {T_x \cong T_y}$. (Any isomorphism between $x,y \in \hat{\C}$ can be lifted to an isomorphism of $\seqo$ into itself respecting the equivalence relations $E_x$ and $E_y$, and then be naturally extended to an isomorphism between $T_x$ and $T_y$. Conversely, from any isomorphism between $T_x$ and $T_y$ one can recover by a back and forth argument an isomorphism between $x$ and $y$.) Let $j_\L \colon S_\infty \times \Mod(\L) \to \Mod(\L)$ be the standard logic action of $S_\infty$ on $\Mod(\L)$. Arguing as in the proof of \cite[Theorem 4.1]{friedmanmottoros}, we can then find a Borel $B \subseteq S_\infty$
such that the map $h \colon (x,b) \mapsto j_\L(b,T_x)$
defined on $\hat{\C} \times B$ is injective and for every $x \in \hat{\C}$ and
$q \in S_\infty$ there are $x' \in \hat{\C}$ and $b \in B$ such that
$j_\L(q,T_x) = j_\L(b,T_{x'})$: therefore the range of this map is Borel and
coincides with the saturation under isomorphism of $\{ T_x \mid x \in \hat{\C}\}$, i.e.\ it is an $\Linf$-elementary class $\C$. Moreover
${\cong_\C} \sim_B {\cong_{\hat{\C}}}$, the equivalence being witnessed by
the Borel map $x \mapsto T_x$ and, for the other direction, by the composition of $h^{-1}$ with the projection on the first coordinate.

It remains to prove that $T_x \sqsubseteq T_y$ for every $x,y
\in \hat{\C}$. This can be easily done by first constructing an embedding of
$T_x \cap \seqo$ into $T_x \cap \{ s \in \seqo \mid |s| \text{ is
  even}\}$ (use the fact that for every $s, t \in \seqo$ there is $t
\subseteq v \in \seqo$ such that $|v|$ is even and $\rp(s^\EE) =
\rp(v^\EE)$), and then extending it to $T_x$ using the fact that each
$s$ of even length has always a successor not in $\seqo$. 
\end{proof}

We will now discuss the case in which ${\id(\RR)} \leq_B
E$.
Recall
from \cite{friedmanmottoros}
that a \emph{combinatorial tree} is a connected acyclic graph, while
an \emph{ordered} combinatorial tree is a combinatorial tree with
an extra transitive (binary) relation defined on its nodes. We  need
the following result from \cite[Theorems 3.3 and 3.5]{friedmanmottoros}. 

\begin{theorem}[\cite{friedmanmottoros}] \label{theorfrimot}
For every analytic quasi-order $R$, there is an $\Linf$-elementary
class $\C$ consisting of ordered combinatorial trees whose order
relation is a strict well-founded order (so that, in particular, it is irreflexive)  such that
${\sqsubseteq_\C}  \sim_B R$ and ${\cong_\C} \simeq_{cB} \id(\RR)$ (and moreover $E_R \simeq_{cB} {\equiv_\C}$).
\end{theorem}

\begin{theorem}\label{theorRcompatible}
Let $E$ be a quasi-isomorphism relation on the standard Borel space $X$ such that $\id(\RR) \leq_B E$, and $R$ be
an arbitrary analytic quasi-order on the standard Borel space $Y$. Then there is an
$\Linf$-elementary class $\C$ such that $E \sim_B
{\cong_\C}$ and $R \sim_B {\sqsubseteq_\C}$ (in fact, if $E$ itself is an isomorphism
 relation then $E \simeq_{cB} {\cong_\C}$ and $E_R \simeq_{cB} {\equiv_\C}$).
\end{theorem}

\begin{proof}
 We can assume $X = Y =  \RR$. Let $\C'$ be given by applying Theorem
 \ref{theorfrimot} to $R$, so that ${\sqsubseteq_{\C'}} \sim_B R$ and
 ${\cong_{\C'}} \simeq_{cB} \id(\RR)$, and let $\fhi_0 \colon \C' \to \RR$ and $\psi_0 \colon \RR \to \C'$
 witness the 
classwise Borel isomorphism. Then apply Theorem \ref{theorequivsimple} to
 $E$ to get an $\Linf$-elementary class $\C''$ such that
 $ {\cong_{\C''}} \sim_B E$ and ${\equiv_{\C''}} \sim_B
 \id(1)$. Since $\id(\RR) \leq_B E$, $\cong_{\C''}$
 essentially refines $\id(\RR)$ by Proposition \ref{propessentiallyrefines}, so let $E'
 \supseteq {\cong_{\C''}}$ be a Borel equivalence
 relation on $\C''$ which is classwise Borel isomorphic to $\id(\RR)$, and let 
 $\fhi_1 \colon \C'' \to \RR$ and $\psi_1 \colon \RR \to \C''$
 be witnesses 
 of this fact. Notice that $\psi_1 \circ \fhi_0 \colon \C' \to \C''$
 and $\psi_0 \circ \fhi_1 \colon \C'' \to \C'$ witness ${\cong_{\C'}}
 \simeq_{cB} E'$.
Now consider the set 
\[ W
 = \{ (x,z,g) \in \C' \times \C'' \times G
 \mid \psi_0(\fhi_1(z)) \cong_{\C'} x \},\]
 where $G$ is the closed subset of $S_\infty$ consisting of those $g$
 such that for all $n,m \in \omega$, if $n,m$ have the same parity then
 $n \leq m \iff g(n) \leq g(m)$. Notice that $W$ is Borel because
 $\cong_{\C'}$ is a Borel equivalence relation, and define 
 $S$ and $F$ 
 on $W$  by 
\[ (x_1, z_1, g_1) \, S \, (x_2, z_2,  g_2) \iff x_1 \sqsubseteq x_2 \iff x_1 \sqsubseteq_{\C'} x_2, \]
\[ (x_1, z_1, g_1) \, F \, (x_2, z_2, g_2) \iff z_1 \cong z_2 \iff z_1 \cong_{\C''} z_2. \]

Obviously, the projections on the first and on the second coordinate witness, respectively, $S \leq_B {\sqsubseteq_{\C'}}$ and $F \leq_B {\cong_{\C''}}$. Moreover, the Borel map sending $z \in \C''$ to
$(\psi_0(\fhi_1(z)),z,\id)$ (which is an element of $W$ by definition)
witnesses ${\cong_{\C''}} \leq_B F$. Consider now the Borel map $h$ sending $x \in \C'$ to
$(x,\psi_1(\fhi_0(x)),\id)$: since $\psi_0(\fhi_1(\psi_1(\fhi_0(x))))
= \psi_0(\fhi_0(x)) \cong_{\C'} x$, we have that $h(x) \in W$, and obviously
$h$ reduces $\sqsubseteq_{\C'}$ to $S$. Therefore we get $S \sim_B {\sqsubseteq_{\C'}}$ ($\sim_B R$) and $F \sim_B {\cong_{\C''}}$ ($\sim_B E$), and hence it will be enough to find an $\L_{\omega_1
  \omega}$-elementary 
class $\C$ such that ${\sqsubseteq_\C} \sim_B S$ and ${\cong_\C}
\sim_B F$. 
Define the Borel map $f$ from $W$ into the space of $\L$-structures 
on $\omega$ by sending $w = (x,z,g)$ into $j_\L(g, x \oplus
z)$, where 
$x \oplus z$ is the structure on $\omega$
obtained by ``copying'' in the obvious way $x$ on the even
numbers and $z$ on the
odd numbers. Let
$w_1 = (x_1, z_1, g_1)$ and $w_2 = (x_2,z_2,g_2)$
denote arbitrary elements of $W$.

\begin{claim}
 $f$ reduces $S$ to $\sqsubseteq$ and $F$ to $\cong$.
\end{claim}

\begin{proof}[Proof of the Claim]
 Assume first that $w_1 \, S\, w_2$, that is $x_1 \sqsubseteq
 x_2$. Since $z_1 \sqsubseteq z_2$  by the
 choice of $\C''$, we can
 glue these two embeddings into an embedding of $x_1 \oplus
 z_1$ into $x_2 \oplus z_2$, whence
 $f(w_1) \sqsubseteq f(w_2)$. Conversely, if $f(w_1) \sqsubseteq
 f(w_2)$ then $x_1 \oplus z_1 \sqsubseteq x_2
 \oplus z_2$ as well. But any such embedding must send
 elements coming from $x_1$ into elements coming from
 $x_2$, as by Theorems \ref{theorfrimot} and \ref{theorequivsimple}
 these are the unique vertices of, respectively, $x_1 \oplus z_1$ and
 $x_2 \oplus z_2$ which 
 \emph{are not} in order relation with themselves. This
 implies $x_1 \sqsubseteq x_2$, whence $w_1\, S\, w_2$. 

Assume now $w_1\, F\, w_2$, so that $z_1 \cong z_2$. Since $E'
\supseteq {\cong_{\C''}}$ and $\fhi_1$ reduces $E'$ to
$\id(\RR)$, we have that $\fhi_1(z_1) = \fhi_1(z_2)$, so that $x_1 \cong 
\psi_0(\fhi_1(z_1)) = \psi_0(\fhi_1(z_2)) \cong x_2$. Therefore one
can glue these 
 isomorphisms to witness $x_1 \oplus z_1 \cong
x_2 \oplus z_2$, whence $f(w_1) \cong
f(w_2)$. Conversely, assume that $f(w_1) \cong f(w_2)$, so that in
particular $x_1 \oplus z_1 \cong x_2 \oplus
z_2$: since any isomorphism witnessing this fact must again
map elements coming from $z_1$ into elements coming from $z_2$ (as
these are the unique elements of the corresponding structure which \emph{are}
in order relation with themselves), from such an
isomorphism one can recover an isomorphism between $z_1$ and
$z_2$, whence $w_1 \, F \, w_2$. 
\renewcommand{\qedsymbol}{$\square$ \textit{Claim}} 
\end{proof}

\begin{claim}\label{claiminj}
 $f$ is injective and ${\rm range}(f)$ is saturated.
.
\end{claim}

\begin{proof}[Proof of the Claim]
Assume first that $f(w_1) = f(w_2)$, and observe that for every $h_1,h_2 \in
G$, if $ \{ h_1(2n+1) \mid n \in \omega\} = \{ h_2(2n+1) \mid n \in
\omega \}$ then $h_1 = h_2$. Since $k = g_i(2n+1) \iff k$ is in order relation with itself ($i=1,2$, $n,k \in \omega$) by Theorems \ref{theorfrimot} and
\ref{theorequivsimple} and the definition of $f$, from $f(w_1) = 
f(w_2)$ we get $ \{ g_1(2n+1) \mid n \in \omega\} = \{ g_2(2n+1) \mid
n \in \omega \}$, and hence we can conclude $g_1 = g_2$. But this
implies $x_1 \oplus z_1 = x_2 \oplus
z_2$, whence $x_1 =  x_2$ and $z_1 =
z_2$.  

For the second part, it is enough to show that ${\rm range}(f)$ is the saturation of $\{
\psi_0(\fhi_1(z)) \oplus z \mid z \in \C''  \}$. One direction is
obvious. For the other direction, note that for each $h \in S_\infty$
there are $g \in G$ and $p,q \in S_\infty$ such that $h(2n) =
g(2p(n))$ and $h(2n+1) = g(2q(n)+1)$ for every $n \in \omega$, so that
$j_\L(h, \psi_0(\fhi_1(z)) \oplus z) = j_\L(g, x' \oplus
z')$, where $x' = j_\L(p,\psi_0(\fhi_1(z))) \in \C'$ and $z' = j_\L(q,z) \in
\C''$. But since $z' \cong 
z$ and $E' \supseteq {\cong_{\C''}}$, we have $\fhi_1(z) =
\fhi_1(z')$, so that $x' \cong_{\C'} \psi_0(\fhi_1(z)) =
\psi_0(\fhi_1(z'))$. Therefore $(x',z',g) \in W$ and $j_\L(h, \psi_0(\fhi_1(z)) \oplus z)
= f(w)$, hence we are done.
\renewcommand{\qedsymbol}{$\square$ \textit{Claim}}
\end{proof}

Since $W = {\rm dom}(f)$ is Borel and $f$ is Borel and injective,
then ${\rm range}(f)$ is Borel and $f^{-1}$ is a Borel function 
reducing $\sqsubseteq$ to $S$ and $\cong$ to $F$. Since ${\rm
  range}(f)$ is also saturated then ${\rm range}(f) = \C$ for some
$\L_{\omega_1 \omega}$-elementary class $\C$, therefore we get the
desired result. 
\end{proof}

Using the technique developed in the previous proof,  we can now deal
with the case $E_R \leq_B \id(\omega),E$.

\begin{theorem}\label{theor'}
 For every $1 \leq n \leq \omega$, every analytic quasi-order $R$ such that
 $E_R \sim_B {\id(n)}$, and every  quasi-isomorphism
 relation $E$ such that ${\id(n)} \leq_B E$, there is an
 $\L_{\omega_1 
   \omega}$-elementary class $\C$ such that $E \sim_B {\cong_\C}$ and
 $R \sim_B {\sqsubseteq_\C}$ (in fact, if $E$ itself is an isomorphism
 relation then $E \simeq_{cB} {\cong_\C}$ and $E_R \simeq_{cB} {\equiv_\C}$). 
\end{theorem}

\begin{proof}
The case $n =1$ is Theorem \ref{theorequivsimple}, thus we will consider
just the case
$1 < n \leq \omega$. Apply Theorem \ref{theorfrimot} and let $\C'$ be an
$\Linf$-elementary class such that $R \sim_B
{\sqsubseteq_{\C'}}$ (and ${\cong_{\C'}} \simeq_{cB} \id(\RR)$, so that
$\cong_{\C'}$ is Borel). Moreover, let $\C''$  be an
$\Linf$-elementary class such that $E \sim_B 
{\cong_{\C''}}$ and $\equiv_{\C''} \sim_B {\id(1)}$ (such a $\C''$ exists by
Theorem \ref{theorequivsimple}). Choose pairwise non-isomorphic $z_1,
\dotsc, z_n \in \C''$ (this is possible because ${\id(n)} \leq_B E$): then
$\hat{\C}'' = \C'' \setminus \bigcup \{ [z_i]_{\cong} \mid 1 \leq i 
< n  \}$ is a nonempty Borel invariant set, so it is an
$\Linf$-elementary class. Now choose
$x_1, \dotsc, x_n \in \C'$ to be pairwise
$\equiv_{\C'}$-inequivalent and such that every other $x \in \C'$ is
$\equiv_{\C'}$-equivalent to some of these $x_i$'s (this is possible
since ${\equiv_{\C'}} \sim_B E_R \sim_B {\id(n)}$). Now put 
\[
  W = \{ (x,z,g) \in \C' \times \C'' \times G \mid  (z \in \hat{\C}'' \wedge x \cong_{\C'} x_n) \vee \bigvee\nolimits_{1 \le i < n} ({z \cong z_i} \wedge {x \cong_{\C'} x_i}) \},
\]
where $G$ is defined as in the proof of Theorem
\ref{theorRcompatible}. Letting now $S$, $F$ and $f$ be defined as in that
proof, by (almost) the same argument one gets that $S \sim_B
{\sqsubseteq_{\C'}} \sim_B R$, $F \sim_B {\cong_{\C''}} \sim_B E$, $f$ reduces $S$ to $\sqsubseteq$ and $E$ to $\cong$ (this
is essentially because if $(x,z,g),(x',z',g') \in W$ and $z \cong z'$
then $x \cong x'$), $f$ is injective and ${\rm range}(f)$ is saturated, so
that taking $\C = {\rm range}(f)$ we have the result. 
\end{proof}

We now want to make some comments on possible variations of our main
result. First of all, notice that in both Theorem
\ref{theorRcompatible} and Theorem \ref{theor'} we can also have that
the resulting $\C$ consists of ordered
combinatorial trees. In fact, given $x \in C'$ (where $\C'$ is as in
the proof of Theorem \ref{theorRcompatible}) call \emph{root} of $x$
the least vertex with respect to the (strict well-founded) order
relation on $x$. By inspecting the proof of Theorem \ref{theorfrimot}, it is
easy to check that any embedding (and consequently any isomorphism)
between $x_1,x_2 \in \C'$ must send the 
root of $x_1$ to the root of $x_2$. Moreover, by applying \cite[Theorem
4.1]{friedmanmottoros} to the $\Linf$-elementary class $\C''$ defined
in the proof of Theorem \ref{theorRcompatible}, we get that such
$\C''$ can be assumed to consist of
ordered combinatorial trees with the further property that for each $z
\in \C''$ there is a unique element which is in order relation just
with itself (such element is called \emph{root} of $z$), and that for
every $z_1,z_2 \in \C''$ there is an embedding between them which
sends the root of $z_1$ to the root of $z_2$ (this easily follows from the construction given in \cite[Theorem 4.1]{friedmanmottoros}). Define the Borel set $W$
as above and redefine $f$ to be the Borel function sending  $(x,z,g)
\in W$ to $j_\L(g, x \, \hat{\oplus} \, z)$, where $x 
\, \hat{\oplus}\, z$ is the \emph{ordered combinatorial tree} obtained by
first considering $x \oplus z$ and then linking the root of $x$ to the
root of $z$. It is now straightforward to check that the
proofs of Theorems \ref{theorRcompatible} and \ref{theor'} can be carried out with this new definition of $f$ (by the
properties of $\C'$ and $\C''$ described above).

The second possible variation is given by the fact that we can replace
embeddings with 
homomorphisms and weak-homomorphisms\footnote{A function $f$
between (the domains of) two structures $x,y$ with the same language is said \emph{homomorphism} if it preserves
relations and functions in both directions, and
\emph{weak-homomorphism} if it preservers relations and functions just
from the domain structure to the range structure. In particular,
embeddings coincide with \emph{injective} homomorphisms.}.  This is because
in \cite[Theorem 
3.5]{friedmanmottoros} 
it is proved that  any weak-homomorphism between two elements of the
$\Linf$-elementary class $\C$ given by Theorem \ref{theorfrimot}
 is
automatically an embedding. This fact, together with the trivial
observation that each embedding is, in particular, a
(weak-)homomorphism, shows that the proofs of Theorems
\ref{theorequivsimple}, \ref{theorRcompatible} and \ref{theor'} are
also proofs of 
the analogous results obtained by replacing $\sqsubseteq$ with the analytic quasi-order naturally induced by (weak-)homomorphism.

Finally, in the case of embeddings and homomorphisms we can further replace
the language of ordered graphs $\L$ (which consists of two binary
relation symbols) with the language of graphs $\hat{\L}$, and obtain that the
$\C$ resulting from any of the theorems of this section consists of
graphs. This is because, as already noticed in the introduction,  in
\cite[Theorem 5.5.1]{hodges} it is shown that the
$\Linf$-elementary class $\C$ can be bi-interpreted in an $\hat{\L}_{\omega_1
  \omega}$-elementary class $\hat{\C}$ consisting of connected graphs, and a careful
inspection of the proof shows that both interpretations preserve
the embeddability and the homomorphism relation (in fact one can 
show that for graphs in $\hat{\C}$ each homomorphism
is automatically an embedding). The case of weak-homomorphisms seems
more difficult, as we do not even know if a statement analogous to
Theorem \ref{theorfrimot} holds when replacing the language $\L$
with $\hat{\L}$.

We end this section with a question about the unique possibility left open by the limitations discussed in the introduction and
Theorems \ref{theorRcompatible} and \ref{theor'}.

\begin{question}
What if Vaught's Conjecture is false, $E \sim_B {\cong_\C}$ for some
$\C$ witnessing this failure, and $F$ is such that ${\id(\omega)} <_B
F$ but ${\id(\RR)}
\nleq_B F$ (or, similarly, if $R$ is an analytic quasi-order such that
${\id(\omega)} <_B E_R$ but ${\id(\RR)} \nleq_B E_R$)? 
\end{question}

Notice that an answer to this question must necessarily employ different techniques, because a careful inspection of the proofs of Theorems \ref{theorRcompatible} and 
\ref{theor'} shows that such kinds of
argument can in general be applied to those pairs $(E,R)$ for which there is an analytic equivalence relation $F \supseteq {\cong_\E}$ (where $\E$ is some
$\hat{\L}_{\omega_1 \omega}$-elementary class such that $E \sim_B {\cong_\E}$) which is Borel
isomorphic to a refinement of $E_R$ and has a Borel 
selector. But the Borel transversal given by such a selector must
 either be countable or contain a perfect subset, and hence $(E,R)$
 must satisfy the hypotheses of Theorem 
 \ref{theor'} in the former case, and of Theorem
 \ref{theorRcompatible} in the other one.

\section{Replacing embeddings with epimorphisms}\label{sectionepi}

Given two countable structures $x,y$, call a function between their
domains \emph{(weak-)epimorphism} if it is a surjective
(weak-)homomorphism, and put $x \preceq^{(w)epi} y$ if and only if
there is a (weak-)epimorphism of $y$ onto $x$ (as usual, we will denote by $\preceq^{(w)epi}_\C$ the
restriction of $\preceq^{(w)epi}$ to the $\Linf$-elementary class $\C$).
It is easy to check that any quasi-order of the form
 $\preceq^{(w)epi}_\C$ is an
analytic quasi-order. We will now prove that given a pair $(E,R)$ as in the
previous section (that is such that either $\id(\RR) \leq_B E$ or $E_R
\leq_B \id(\omega),E$), it is always possible
to produce an $\hat{\L}_{\omega_1 \omega}$-elementary class $\hat{\C}$ such that $E
\sim_B {\cong_{\hat{\C}}}$ and $R \sim_B {\preceq^{wepi}_{\hat{\C}}}$ (this implies the analogous statement for $\Linf$-elementary classes, since it is enough to adjoin the empty order relation to the elements of $\hat{\C}$). We do not know if a similar result holds for the quasi-order $\preceq^{epi}$: in fact, it is still an open problem whether this relation is complete for analytic quasi-orders.

First we have to consider the basic case (which is
analogous to Theorem \ref{theorequivsimple}).

\begin{theorem}\label{theorequivsimpleepi}
 Let $E$ be a quasi-isomorphism
 relation. Then there is an $\Linf$-elementary class
 $\C$ such that $E \sim_B {\cong_\C}$ and
 $z_1 \preceq^{wepi} z_2$ for every $z_1, z_2 \in \C$ (in fact, if $E$ itself is an isomorphism relation then $E \simeq_{cB} {\cong_\C}$).
\end{theorem}

\begin{proof}
Let $\hat{\C}$ be an $\hat{\L}_{\omega_1 \omega}$-elementary class consisting of graphs
such that
${\cong_{\hat{\C}}} \sim_B E$ and $x_1 \sqsubseteq x_2$ for every $x_1,
x_2 \in \C'$ (such class exists by Theorem \ref{theorequivsimple} and the observations following Theorem \ref{theorRcompatible}). We partially modify the construction given in the proof of Theorem \ref{theorequivsimple}, and for $x \in \hat{\C}$ define a set-theoretical tree $\hat{R}_x$
on 
${}^{<\omega}\omega \sqcup \omega$ as follows: consider the tree ${}^{< \omega} \omega$ with the inclusion relation. If $s \in {}^{<\omega}\omega \setminus \{ \emptyset \}$
  is such that $s(i) = 0$ for some $i < |s|$ then adjoin to $s$ two distinct terminal successors taken from $\omega$, while if  $\emptyset \neq s \in {}^{<\omega}(\omega \setminus \{ 0 \})$ then adjoin to $s$ a unique terminal successor taken from $\omega$ just in case $n-1$ and $m-1$, where $\rp(s) = (n,m)$, are linked in the graph $x$ (as in the original argument, at the end of the above
  construction each element of $\omega$ must be the immediate successor of exactly one
  element of ${}^{< \omega} \omega$).
Now construct an ordered set-theoretical tree $R_x$ by adjoining to
$\hat{R}_x$ the equivalence relation $E_x$ defined in the proof of Theorem \ref{theorequivsimple}.

First check that ${x \cong y} \iff {R_x \cong R_y}$. For one
direction, if $f$ is an isomorphism between $x$ and $y$, first define
$f'(0) = 0$ and $f'(n+1) = f(n)+1$, lift $f'$ to an isomorphism of
${}^{<\omega} \omega$ into itself (which necessarily respect the
equivalence relations 
$E_x$ and $E_y$), and then extend such isomorphism in the obvious way
to an 
isomorphism of $R_x$ and $R_y$. For the other direction, given an
isomorphism $g$ between $R_x$ and $R_y$ recover by a back and forth
argument an isomorphism between $x$ and $y$ --- it is enough to use
the fact  that a sequence $s \in {}^{< \omega}(\omega \setminus \{
0 \})$ 
must be sent into an element of ${}^{< \omega}(\omega \setminus \{ 0
\})$ because such an $s$ can have at most one terminal successor, while every sequence which contains a $0$ has two distinct 
terminal successors. 

Then argue as in the proof of Theorem \ref{theorequivsimple} to
show that the saturation under isomorphism of $\{ R_x \mid x
\in \hat{\C} \}$ forms an $\Linf$-elementary class $\C$ such that ${\cong_\C} \sim_B {\cong_{\hat{\C}}}$,
so that it remains only to prove that for $z_1,z_2 \in \C$ one has $z_1
\preceq^{wepi} z_2$. Clearly it is enough to show that
for $x,y \in 
\hat{\C}$ there is a weak-epimorphism from $R_y$ onto $R_x$. Let $f$ be
an embedding from $x$ into $y$, and define $g$ by letting $g(0) = 0$,
$g(n+1) = m+1$ if $f(m) = n$, and $g(n+1) = 0$ otherwise. Now lift
coordinatewise the function $g$ to a surjection $\hat{g}$ from ${}^{<
  \omega} 
\omega$ onto itself. Note that if $s = \langle s_0, \dotsc s_n \rangle
\in {}^{< \omega}(\omega 
\setminus \{0\})$ has a terminal successor in $R_x$, then
the 
sequence $t = \langle f(s_0 - 1)+1, \dotsc f(s_n - 1)+1 \rangle$ has a
terminal successor  in $R_y$ and is such that $\hat{g}(t) = s$,
while if $s = \langle s_0, \dotsc, s_n \rangle$ is such that $s_i = 0$
for some $i  \leq n$  then the sequence $t = \langle t_0, \dotsc, t_n
\rangle$ defined by $t_i = 0$ if $s_i = 0$ and $t_i = f(s_i - 1) + 1$
otherwise is such that $\hat{g}(t) = s$ and both $s$ and $t$ have
exactly two distinct terminal successors in $R_x$ and $R_y$,
respectively. Therefore one can extend $\hat{g}$ in the obvious way to
a 
weak-epimorphism $h$ from $R_y$ onto $R_x$. 
\end{proof}

\begin{corollary}\label{corequivsimpleepi}
For every quasi-isomorphism
 relation $E$, there is an $\hat{\L}_{\omega_1 \omega}$-elementary class
 $\hat{\C}$ consisting of connected graphs such that $E \sim_B {\cong_{\hat{\C}}}$ and
 $z_1 \preceq^{wepi} z_2$ for every $z_1, z_2 \in \hat{\C}$ (in fact, if $E$ itself is an isomorphism relation then $E \simeq_{cB} {\cong_{\hat{\C}}}$).
\end{corollary}

\begin{proof}
Assume $\L = \{ P_0, P_1\}$, with $P_0,P_1$ binary relation symbols.
Then it is enough to (bi-)interpret the class $\C$ given by the proof of Theorem \ref{theorequivsimple} in a class $\hat{\C}$ of connected graphs as explained in \cite[Theorem  5.5.1]{hodges}, and check that 
even if in general such interpretation does not preserve $\preceq^{wepi}$, it is still true that if 
there is a weak-epimorphism $h$ from  $y \in \C$ onto $x \in \C$ \emph{such that
  for every 
$n,m$ in $P_i$-relation in $x$ ($i=0,1$) there are $l \in h^{-1}(n),k \in h^{-1}(m)$ which are in $P_i$-relation in $y$} (which is the case
for the elements of the class $\C$ constructed above) then there is a weak-epimorphism from
the interpretation of $y$ onto the interpretation of $x$.
\end{proof}

To prove the statement analogous to Theorems \ref{theorRcompatible}
and \ref{theor'}, we need to replace Theorem
\ref{theorfrimot} with the following result obtained in \cite[Section
5.1]{mottoroscamerlo}. 

\begin{theorem}[\cite{mottoroscamerlo}]\label{theormottoroscamerlo}
For every analytic quasi-order $R$ there is an $\hat{\L}_{\omega_1
  \omega}$-elementary class $\hat{\C}$ consisting of connected graphs such that
${\preceq^{wepi}_{\hat{\C}}} \sim_B R$ and ${\cong_{\hat{\C}}} \simeq_{cB} \id(\RR)$ (and moreover ${\approx^{wepi}_{\hat{\C}}} \simeq_{cB} E_R$, where $\approx^{wepi}_{\hat{\C}}$ is the equivalence relation associated to $\preceq^{wepi}_{\hat{\C}}$). 
\end{theorem}

We can now repeat the
proofs of Theorems \ref{theorRcompatible} and \ref{theor'} and show
the following.

\begin{theorem}\label{theorEcompatibleepi}
Let $E$ be a quasi-isomorphism relation and 
$R$ be 
an arbitrary analytic quasi-order such that either $\id(\RR) \leq_B E$
or $E_R \leq_B \id(\omega), E$. Then
there is an $\hat{\L}_{\omega_1 \omega}$-elementary class $\hat{\C}$ such that
$E \sim_B  {\cong_{\hat{\C}}}$ and $R \sim_B
{\preceq^{wepi}_{\hat{\C}}}$ (in fact, if $E$ itself is an isomorphism
 relation then $E \simeq_{cB} {\cong_{\hat{\C}}}$ and $E_R \simeq_{cB} {\approx^{wepi}_{\hat{\C}}}$).
\end{theorem}

\begin{proof}
Carry out the proofs of Theorems \ref{theorRcompatible} and \ref{theor'} by replacing the original $\C'$ and $\C''$ with the classes $\hat{\C}'$ and $\hat{\C}''$ obtained by applying, respectively, Theorem \ref{theormottoroscamerlo} to $R$ and Corollary \ref{corequivsimpleepi} to $E$, and by redefining the quasi-order $S$ on $W$ using $\preceq^{wepi}$ instead of $\sqsubseteq$. By inspecting those proofs, it is enough to show that 

\begin{enumerate}[(a)]
 \item  given $w_1 = (x_1,z_1,g_1), w_2 = (x_2,z_2,g_2) \in W$, any weak-epimorphism $h$ of $f(w_2) = j_{\hat{\L}}(g_2, x_2 \oplus z_2)$ onto $f(w_1) = j_{\hat{\L}}(g_1, x_1 \oplus z_1)$ (hence, in particular, any isomorphism between $f(w_2)$ and $f(w_1)$) must send elements coming from $x_2$ (resp.\ $z_2$) into elements coming from $x_1$ (resp.\ $z_1$), and 

\item the vertices of $f(w_i)$ coming from $x_i$ (and hence also the vertices of $f(w_i)$ coming from $z_i$, where $i=1,2$) can be recognized in an intrinsic way, i.e.\ using a property of the structure $f(w_i)$ which does not use any knowledge about $w_i$. 
\end{enumerate}

By the proofs of Theorem \ref{theormottoroscamerlo} and \cite[Theorem 5.5.1]{hodges}, the vertices of $f(w_i)$ coming from $x_i$ ($i=1,2$) are uniquely determined by the property of belonging to a nontrivial (i.e.\ of size $\geq 3$) clique (this gives part (b) above). But it is easy to check that $h$ must send each nontrivial clique into a clique of greater or equal size (so that vertices of $f(w_2)$ coming from $x_2$ are sent by $h$ into vertices of $f(w_1)$ coming from $x_1$), and that, consequently, vertices of $f(w_2)$ coming from $z_2$ are sent by $h$ into vertices of $f(w_1)$ coming from $z_1$ because $h$ is surjective and must send connected subgraphs of $f(w_2)$ to connected subgraphs of $f(w_1)$: this gives part (a) and concludes our proof.
\end{proof}

\end{document}